\documentclass[reqno]{amsart}
\makeatletter                                                  %
\def\th@plain{\slshape}                                        %
\makeatother                                                   %

\usepackage[italian,english]{babel} 
\usepackage[utf8x]{inputenc}
\usepackage[T1]{fontenc} 
\usepackage{indentfirst} 
\usepackage{amsmath,amsfonts,amssymb,amsthm}
\usepackage[all]{xy}
\usepackage{xcolor}
\usepackage{graphicx}
\usepackage{geometry}
\usepackage{caption}
\usepackage{subcaption}
\usepackage[font={small}, labelfont=bf]{caption}[2004/07/16]
\usepackage{url,mathrsfs}
\usepackage{tikz}

\newcommand*{\diff}{\mathop{}\!\mathrm{d}}

\theoremstyle{plain}
\newtheorem{thm}{Theorem}
\newtheorem{prop}[thm]{Proposition}
\newtheorem{cor}[thm]{Corollary}
\newtheorem{lemma}[thm]{Lemma}

\theoremstyle{definition}
\newtheorem{defin}[thm]{Definition}

\theoremstyle{remark}
\newtheorem{remark}{Remark}

\newcommand{\R}{\mathbb{R}}

\newcommand{\N}{\mathbb{N}}

\newcommand{\norm}[1]{\left\lVert#1\right\rVert}

\newcommand{\modulo}[1]{\left\lvert#1\right\rvert}
\newcommand{\newword}[1]{\textsl{#1}}
\newcommand{\sltre}{\SL(3,\R)}
\newcommand{\lietre}{\mathfrak{sl}(3,\R)}

\DeclareMathOperator{\spn}{span}

\DeclareMathOperator{\SL}{SL}

\begin{document}

\bibliographystyle{plain}

\sloppy

\title[Parabolic perturbations of unipotent flows]{Parabolic perturbations of unipotent flows on compact quotients of~$\sltre$}

\author[D.~Ravotti]{Davide Ravotti}
\address{School of Mathematics\\
University of Bristol\\
University Walk\\
BS8 1TW Bristol, UK}
\email{davide.ravotti@bristol.ac.uk}

\begin{abstract}
We consider a family of smooth perturbations of unipotent flows on compact quotients of SL$(3,\mathbb{R})$ which are not time-changes. More precisely, given a unipotent vector field, we perturb it by adding a non-constant component in a commuting direction. We prove that, if the resulting flow preserves a measure equivalent to Haar, then it is parabolic and mixing. The proof is based on a geometric shearing mechanism together with a non-homogeneous version of Mautner Phenomenon for homogeneous flows. Moreover, we characterize smoothly trivial perturbations and we relate the existence of non-trivial perturbations to the failure of cocycle rigidity of parabolic actions in SL$(3,\mathbb{R})$.
\end{abstract}

\thanks{\emph{2010 Math.~Subj.~Class.}: 37A25, 37C40.}

\maketitle

\section{Introduction}

In this paper, we give a contribution to the ergodic theory of parabolic flows, namely flows for which nearby orbits diverge polynomially in time (see Definition \ref{def1}). Classical examples of parabolic flows are given by horocycle flows on compact negatively curved surfaces, more in general by unipotent flows on semisimple Lie groups, and by nilflows on nilmanifolds. Although the homogeneous case is well-understood, very little is known for general smooth parabolic flows. An important class of non-homogeneous parabolic flows is given by perturbations of homogeneous ones; the simplest of which are time-changes. A smooth time-change of a flow is obtained by varying smoothly the speed of the points while keeping the same trajectories. 

It is natural to ask which ergodic properties persist after performing a time-change.
In the case of the horocycle flow, mixing and mixing of all orders for all time-changes which satisfy a mild differentiability condition were proved by Marcus in \cite{marcus:horocycle, marcus:horocycle2}. More recently, Tiedra de Aldecoa \cite{tiedra:spectrum} and Forni and Ulcigrai \cite{forniulcigrai:timechanges} independently showed that generic time-changes have absolutely continuous spectrum (in the latter paper, the authors show in addition that the spectrum is equivalent to Lebesgue; see also the result by Simonelli \cite{simonelli:spectrum}). The case of time-changes of nilflows has been treated by Avila, Forni and Ulcigrai in \cite{afu:heisenberg} for the Heisenberg group and by the author in \cite{ravotti:nilflow} for a class of higher-dimensional and higher-step nilpotent groups. 

Here, we investigate the ergodic properties of a class of parabolic perturbations of unipotent flows on compact quotients of $\sltre$ which \newword{are not} time-changes or skew-product constructions; to the best of our knowledge, this is the first such example. We consider a unipotent vector field $U$ on a compact homogeneous manifold $\mathcal{M} = \Gamma \backslash \sltre$ and we add a non-constant component in a transverse direction $Z$ commuting with $U$. More precisely, given a smooth function $\beta \colon \mathcal{M} \to \R$, we consider the flow $\{ \widetilde{h}_t\}_{t \in \R}$ induced by the vector field $\widetilde{U}=U + \beta Z$, see \S\ref{s:preliminaries}. We prove that, if $\{ \widetilde{h}_t\}_{t \in \R}$ preserves a measure equivalent to Haar, then it is ergodic and, in fact, mixing. The key observation is that there exists a vector field $W$ such that the Lie derivative $\mathscr{L}_{\widetilde{U}}(W)$ is parallel to $Z$. Roughly speaking, this means that short segments in direction $W$ get sheared along the direction $Z$ when flown via $\{ \widetilde{h}_t\}_{t \in \R}$. Since the flow in direction $Z$ is ergodic, such segments become equidistributed.

We chose to work with SL$(3,\mathbb{R})$ in order to provide a concrete example and making the computations explicit, but we believe it should be possible to carry out a similar approach for suitable perturbations of unipotent flows in compact quotients of all semisimple Lie groups, see Remark \ref{rk:ingen}. 

In our proof, we exploit the geometrical information given by computing the Lie brackets $[\widetilde{U},W]$ (see \S\ref{s:ODE}) and we employ smooth analogues of well-known homogeneous arguments. The main difficulty in this setting is to prove that $\{ \widetilde{h}_t\}_{t \in \R}$ is ergodic. We remark that this is not an issue in the case of time-changes, since they preserve the orbit structure and they admit an invariant measure equivalent to Haar; hence they are ergodic. The proof of ergodicity for the perturbed flow $\{ \widetilde{h}_t\}_{t \in \R}$ can be seen as a non-homogeneous version of Mautner Phenomenon and we believe it is interesting in its own right, see \S\ref{s:ergodicity}. In order to help the reader in following the arguments, we postpone the proof of an auxiliary proposition to~\S\ref{s:lemmas}. The proof of mixing 
is presented in~\S\ref{s:ergodicity}. In \S\ref{s:ppcorig}, we relate the existence of perturbations which are not smoothly trivial (i.e., which are not $\mathscr{C}^1$-conjugated to the original flow) to the failure of cocycle rigidity for parabolic actions in $\sltre$, see Theorem \ref{thm:cocrig}, whose proof is contained in \S\ref{s:pp7}.

\section{Preliminaries}\label{s:preliminaries}

Let $\mathcal{M} = \Gamma \backslash \sltre$  be a compact connected homogeneous manifold and let $\omega$ be the differential form on $\mathcal{M}$ inducing the normalised Haar measure. 
The Lie algebra $\lietre$ of $\sltre$ consists of $3 \times 3$ matrices $X$ with zero trace; we identify it with the set of left-invariant vector fields on $\mathcal{M}$ (see, e.g., \cite[Proposition 1.72]{gallot:riemannian}). 

Denote by $E_{i,j}$ the $3 \times 3$ matrix with 1 in position $(i,j)$ and 0 elsewhere. We decompose 
$$
\lietre = \mathfrak{n}^{\text{tr}} \oplus \mathfrak{a} \oplus \mathfrak{n},
$$
where 
$$
\mathfrak{a} = \spn \left\{ \frac{1}{2}( E_{1,1} - E_{2,2}), \frac{1}{2}( E_{2,2} - E_{3,3}) \right\}
$$ 
is a maximal abelian subalgebra, and 
$$
\mathfrak{n} = \spn \{ E_{1,2}, E_{2,3}, E_{1,3} \} \text{\ \ \ and\ \ \ } \mathfrak{n}^{\text{tr}} = \spn \{E_{3,1} , E_{2,1}, E_{3,2} \}
$$
are nilpotent subalgebras: the only nontrivial brackets in $\mathfrak{n}$ and $\mathfrak{n}^{\text{tr}}$ are $[E_{1,2},E_{2,3}]=E_{1,3}$ and $[E_{3,2},E_{2,1}]=E_{3,1}$ respectively. 
More generally, the commutation relations in $\lietre$ are given by
$$
[E_{i,j}, E_{l,m}]= \delta_{j,l} E_{i,m} - \delta_{m,i} E_{l,j}
$$
where $\delta_{i,j}$ is the Kronecker delta.
We remark that the centre $\mathfrak{z}(\mathfrak{n})$ of $\mathfrak{n}$ is 1-dimensional and is generated by $Z:= E_{1,3}$. 
Let 
\begin{equation}\label{eq:thebasis}
\mathscr{B}= \left\{ E_{3,1} , E_{2,1}, E_{3,2} ,   \frac{1}{2}( E_{1,1} - E_{2,2}), \frac{1}{2}( E_{2,2} - E_{3,3})  , E_{1,2}, E_{2,3}, E_{1,3} \right\}
\end{equation}
be the basis of $\lietre$ associated to the decomposition above: it is a frame on $\mathcal{M}$, namely a set of vector fields which gives a basis of the tangent space $T_p\mathcal{M}$ at every point $p \in \mathcal{M}$.

For any vector field $X$ (not necessary left-invariant) on $\mathcal{M}$, we denote by $\{ \varphi^X_t \}_{t \in \R}$ the induced flow. If $X \in \lietre$, we have an explicit formula for $\{ \varphi^X_t \}_{t \in \R}$, namely for all $p = \Gamma g \in \mathcal{M}$,
$$
\varphi^X_t (\Gamma g) = \Gamma g \exp(tX).
$$
In other words, the flow $\{ \varphi^X_t \}_{t \in \R}$ is given by the right-action on $\mathcal{M}$ of the one-parameter subgroup $\{ \exp(tX) : t \in \R \}$. By the Howe-Moore Ergodicity Theorem, every noncompact subgroup as above acts ergodically on $\mathcal{M}$. 

If $X \in \mathfrak{n}$, then $\{ \exp(tX) : t \in \R \}$ consists of unipotent matrices, hence $\{ \varphi^X_t \}_{t \in \R}$ is said to be a unipotent flow and $X$ a unipotent vector field. Unipotent flows are mixing of all orders and have countable Lebesgue spectrum, see \cite{mozes} and \cite{brezinmoore}. Moreover, a great amount of work has been carried out in investigating their ergodic invariant measures, from the results by Furstenberg \cite{furst} and Dani \cite{dani1} for the classical horocycle flow, by Dani and Margulis \cite{danimargulis} for generic unipotent flows in quotients of $\sltre$, to the celebrated theorems of Ratner \cite{ratner1, ratner2, ratner3}; see also the generalizations to $p$-adic groups by Ratner \cite{ratner4} and by Margulis and Tomanov \cite{margulistomanov}.

To prove these measure rigidity results, one crucially uses that nearby orbits diverge polynomially in time. One version of this property is encoded in the following definition.
\begin{defin}\label{def1}
We will say that the smooth flow $\{ \varphi_t \}_{t \in \R}$ is \newword{parabolic} if there exists $n \in \N$ such that 
$$
\norm{D \varphi_t}_{\infty} = O(|t|^n) \text{\ \ \ as } t \to 0,
$$
where $D\varphi_t$ is the differential of $\varphi_t$.
\end{defin}

Fix a non-zero unipotent vector field 
\begin{equation}\label{eq:theu}
U = c_{1,2}E_{1,2} + c_{2,3} E_{2,3} + c_{1,3} E_{1,3}  \in \mathfrak{n} \setminus \{0\},
\end{equation}
and consider a sufficiently small $\mathscr{C}^1$-function $\beta \colon \mathcal{M} \to \R$ (how small will be determined later, see \eqref{eq:conditionbeta} below). We investigate the properties of the flow $\{ \widetilde{h}_t \}_{t \in \R}$ induced by the non-constant perturbation $\widetilde{U} = U + \beta Z$ of $U$. If $U$ is parallel to $Z$, then the flow $\{ \widetilde{h}_t \}_{t \in \R}$ is a time-change of $\{ \varphi^Z_t \}_{t \in \R}$. 
In this case, it is known that the spectrum is absolutely continuous with respect to Lebesgue; in particular, the time-change is mixing, see \cite{simonelli:spectrum}.
In this paper, we will assume that $U \notin \mathfrak{z}(\mathfrak{n}) = \R Z$; i.e., we will consider perturbations which do not preserve orbits. In particular, we have to prove that they are ergodic, which constitutes the main difficulty in this set-up.

Since $U \in \mathfrak{n} \setminus \mathfrak{z}(\mathfrak{n})$, we have that $c_{1,2}^2 + c_{2,3}^2 > 0$; hence we can choose a unipotent $W \in \mathscr{B}$ such that $[U,W]= -c Z$ for some $c \neq 0$ (e.g., if $c_{1,2} \neq 0$, take $W = E_{2,3}$ so that $[U,W] = c_{1,2}Z$). We assume that 
\begin{equation}\label{eq:conditionbeta}
\norm{W\beta}_{\infty} < \modulo{c}. 
\end{equation}
The result we prove is the following.
\begin{thm}\label{thm:1}
Suppose that the flow $\{ \widetilde{h}_t \}_{t \in \R}$ induced by $\widetilde{U} = U + \beta Z$ satisfies \eqref{eq:conditionbeta} and preserves a measure $\widetilde{\omega} = \lambda \omega$ equivalent to Haar, with a smooth density $\lambda \in \mathscr{C}^1(\mathcal{M})$. Then, $\{ \widetilde{h}_t \}_{t \in \R}$ is parabolic, namely $\lVert D \widetilde{h}_t \rVert_{\infty}=O(|t|^4)$, ergodic and mixing.
\end{thm}

In the following section, we explain and comment on the assumption of Theorem \ref{thm:1} on the existence of a smooth equivalent invariant measure and we point out the implications to our context of the failure of cocycle rigidity of parabolic action in SL$(3,\R)$, proved by Wang in \cite{wang}.
In particular, we show that there exist perturbations $\{ \widetilde{h}_t \}_{t \in \R}$ which preserve a smooth equivalent measure and are not smoothly isomorphic to the original homogeneous flow $\{ \varphi^U_t \}_{t \in \R}$.

\begin{remark}\label{rk:ingen}
The properties of the vector fields $U, Z \in \mathfrak{n}$ that we will exploit in the proof of Theorem \ref{thm:1} are the following:
\begin{enumerate}
\item the flow in direction $Z$ is ergodic,
\item $U$ and $Z$ can be included in a \newword{Heisenberg triple} $\{U,W, Z\}$, namely there exists $W \in \mathfrak{n}$ such that $[U,W] = Z$ and $[U,Z] = [W,Z] = 0$. 
\end{enumerate}
We thus believe that Theorem \ref{thm:1} holds in more general settings than the case of SL$(3,\mathbb{R})$. 
For example, consider a real semisimple Lie algebra $\mathfrak{g}$ and let $\mathfrak{g} = \mathfrak{k} \oplus \mathfrak{a} \oplus \mathfrak{n}$ be a Iwasawa decomposition, where $\mathfrak{n}$ is a $k$-step nilpotent subalgebra. 
Then, one can show that for almost every $U \in \mathfrak{n}$ there exists $W \in \mathfrak{n}$ and $Z \in \mathfrak{n}^{(k)}$ such that $\{U,W,Z\}$ is a Heisenberg triple.
Therefore, it is possible to generalize the proof of Theorem \ref{thm:1} to show that, also in this set-up, any flow induced by a vector field of the form $U+\beta Z$, with $\norm{\beta}_{\mathscr{C}^1}$ sufficiently small, which preserves a measure equivalent to Haar is parabolic and mixing.
\end{remark}


\section{Trivial perturbations and cocycle rigidity}\label{s:ppcorig}

Let us consider a perturbation $\widetilde{U} = U + \beta Z$, with $\beta \in \mathscr{C}^1(\mathcal{M})$. 
Up to replacing $U$ with the homogeneous vector field $U + (\int_{\mathcal{M}} \beta\ \omega) Z$, we can assume that $\int_{\mathcal{M}} \beta\ \omega =0$.
We assume that there exists a $\mathscr{C}^1$-density function $\lambda \colon \mathcal{M} \to \R_{>0}$ such that the flow $\{ \widetilde{h}_t \}_{t \in \R}$ preserves the measure $\lambda \omega$ equivalent to Haar.
While this was obvious in the case of time-changes, see e.g.~\cite[\S2]{forniulcigrai:timechanges}, in our case it translates in the following condition
$$
0=\mathscr{L}_{\widetilde{U}}(\lambda \omega) = \diff (\widetilde{U} \lrcorner \lambda \omega) = \diff (\lambda U \lrcorner\ \omega + \beta \lambda Z \lrcorner\ \omega) = (U\lambda + Z(\beta \lambda)) \omega ,
$$
where $\mathscr{L}_{\widetilde{U}}(\lambda \omega) $ denotes the Lie derivative of $\lambda \omega$ with respect to $\widetilde{U}$ and $ \lrcorner$ is the contraction operator. Therefore, there exists a smooth equivalent invariant measure $\lambda \omega$ if and only if $\lambda$ is a solution to the following equation
\begin{equation}\label{eq:vol}
U\lambda + Z(\beta \lambda) = \widetilde{U}\lambda + \lambda Z\beta = 0, \text{\ \ \ with\ } \lambda >0.
\end{equation}
\begin{remark}\label{rk:comm}
The assumption of Theorem \ref{thm:1} is equivalent to the fact that there exists a time-change of the flow $\{\varphi^Z_t\}_{t \in \R}$ in direction $Z$ which commutes with $\widetilde{h}_t$. Indeed, if we set $\widetilde{Z} = (1/\lambda) Z$, we have 
$$
\mathscr{L}_{\widetilde{U}} (\widetilde{Z} ) = \left[ \widetilde{U},  \frac{1}{\lambda}Z \right]= \widetilde{U}\Big( \frac{1}{\lambda}\Big) Z - \frac{Z \beta}{\lambda} Z = -\frac{1}{\lambda^2}(\widetilde{U}\lambda + \lambda Z\beta)Z,
$$
which equals $0$ if and only if \eqref{eq:vol} holds. If this is the case, for every $r,t \in \R$, we have $ \widetilde{h}_t \circ \varphi^{\widetilde{Z} }_r = \varphi^{\widetilde{Z} }_r\circ \widetilde{h}_t$.
\end{remark}
Let us consider the equation 
\begin{equation}\label{eq:ufzg}
Uf + Zg = 0, \text{\ \ \ with\ } \int_{\mathcal{M}} f\ \omega = \int_{\mathcal{M}} g\ \omega = 0.
\end{equation}
We say that any smooth solution to \eqref{eq:ufzg} is a smooth cocycle over the abelian action of $U$ and $Z$, or a \newword{smooth $(U,Z)$-cocycle} for short.
In the language of foliated differential forms, a smooth $(U,Z)$-cocycle is a smooth closed foliated 1-form $\Omega = -g \diff U + f \diff Z$ with respect to the foliation generated by $U$ and $Z$. 

\begin{lemma}\label{thm:trivialtrivial}
Smooth measure-preserving perturbations $\widetilde{U}$ are in one-to-one correspondence with smooth $(U,Z)$-cocycles $(f,g)$, with $f > -1$, by 
$$
(\lambda, \beta) \mapsto \left(\lambda -1 , \lambda \beta - \int_{\mathcal{M}} \lambda \beta\ \omega \right) 
$$
\end{lemma}
\begin{proof}
Given a perturbation $\widetilde{U} = U + \beta Z$ preserving the smooth measure $\lambda \omega$, from \eqref{eq:vol} we deduce that $f = \lambda -1$ and $g= \lambda \beta - ( \int \lambda \beta\ \omega) $ are a smooth solution of \eqref{eq:ufzg}, with $f > -1$.
Conversely, let $(f,g)$ be a smooth $(U,Z)$-cocycle with $f>-1$. Then, $\beta = (g+c)/(1+f)$ defines a perturbation $\widetilde{U}$ that preserves the measure $\lambda = 1+f$, where the constant $c \in \R$ is chosen so that $\int \beta \omega =0$. 
\end{proof}

We say that a perturbation $\widetilde{U}$ is \newword{smoothly trivial} if there exists a $\mathscr{C}^1$-diffeomorphism $F \colon \mathcal{M} \to \mathcal{M}$ which conjugates the perturbation $\{ \widetilde{h}_t\}_{t \in \R}$ to the homogeneous flow $\{ \varphi^U_t \}_{t \in \R}$, namely if the push-forward $(F)_{\ast}$ maps $\widetilde{U}$ to $ U$.

\begin{thm}\label{thm:conjug}
The perturbation $\{ \widetilde{h}_t\}_{t \in \R}$ is $\mathscr{C}^1$-conjugated to the homogeneous flow $\{ \varphi^U_t \}_{t \in \R}$ if and only if there exists $w \in \mathscr{C}^{\infty}(\mathcal{M})$, with $Zw > -1$, such that $\beta = \widetilde{U}w$.
\end{thm}

The proof of Theorem \ref{thm:conjug} is presented in \S\ref{s:pp7}. 

\begin{cor}\label{thm:cocrig}
Smoothly trivial perturbations $\widetilde{U}$ are in one-to-one correspondence with $(U,Z)$-cocycles $(f,g)$ of the form $f=Zw >-1$ and $g=-Uw$, for some $w \in \mathscr{C}^{\infty}(\mathcal{M})$.
\end{cor}
\begin{proof}
By Theorem \ref{thm:conjug}, $\widetilde{U}$ is smoothly trivial if and only if there exists $w \in \mathscr{C}^{\infty}(\mathcal{M})$, with $Zw >-1$, such that $\beta = \widetilde{U}(-w) = -Uw - \beta Zw$, if and only if $\beta = -Uw/(1+Zw)$.
By \eqref{eq:vol}, $\widetilde{U}$ preserves the smooth measure with density $\lambda = 1+ Zw$.
By Lemma \ref{thm:trivialtrivial}, the pair $(\lambda, \beta)$ uniquely defines the smooth $(U,Z)$-cocycle $(f,g)$, where $f = \lambda -1 = Zw$ and $g =\lambda \beta - \int_{\mathcal{M}} \lambda \beta\ \omega = -Uw$.
\end{proof}
In view of Corollary \ref{thm:cocrig}, in order to ensure the existence of perturbations $\widetilde{U}$ which are not smoothly conjugate to the original unipotent flow, we need to address the cohomological problem of establishing whether all the solutions to \eqref{eq:ufzg} arise from a common smooth function $w$ or not.  
We say that the action of the commuting vector fields $U$ and $Z$ is \newword{cocycle rigid} if the following holds
\[
\text{if $(f,g)$ is a solution to \eqref{eq:ufzg}, then there exists $w$ such that $f=Zw$ and $g=-Uw$.} \label{eq:CR} \tag{CR}
\]
The question of cocycle rigidity (and related problems) on homogenous spaces has been investigated by several authors in different settings, including, among others, Damjanovic and Katok \cite{damjanovic}, Katok and Spatzier \cite{spatzier} for partially hyperbolic actions, and by Flaminio and Forni \cite{flaminioforni}, Mieczkowski \cite{miec}, Ramirez \cite{ramirez}, and Wang \cite{wang} for parabolic actions.  
For one-dimensional actions, Flaminio, Forni and Rodriguez-Hertz \cite{FFRH} showed that any homogeneous flow has infinitely many linearly independent obstructions to cocycle trivialization, unless it is smoothly isomorphic to a Diophantine linear flow on a torus. 

It turns out that, in general, cocycle rigidity for SL$(3,\R)$ \newword{fails}: Wang showed that, for example, for $U=E_{1,2}$ and some lattice $\Gamma \leqslant G$, there are smooth functions $f,g \in \mathscr{C}^{\infty}(\mathcal{M})$ such that \eqref{eq:ufzg} is satisfied, but the equations $f=Zw$ and $g=-Uw$ have no common solution in $L^2(\mathcal{M})$ (and hence in $\mathscr{C}^{\infty}(\mathcal{M})$), see Theorems 2.5, 2.6 and Remark 2.7 in \cite{wang}. In particular, in our case, there exist perturbations $\widetilde{U}$ that satisfy the assumption of Theorem \ref{thm:1}, and hence are parabolic and mixing, but, by Corollary \ref{thm:cocrig}, are not smoothly trivial, i.e.~they are not $\mathscr{C}^1$-conjugated to the unperturbed homogeneous flow.

\begin{remark}
The problem of establishing whether there exists a measurable isomorphism conjugating $\{ \widetilde{h}_t\}_{t \in \R}$ with $\{ \varphi^U_t \}_{t \in \R}$ remains open, but appears to be a difficult question. 
Indeed, we remark that, in the simpler case of time-changes, the existence of time-changes of the classical horocycle flow which are not measurably conjugated to the horocycle flow itself follows from the rigidity theorem of Ratner \cite{ratner0} and deep results on the classification of invariant distributions and on the deviations from the ergodic averages proved by Flaminio and Forni \cite{flaminioforni}, see, e.g., \cite[\S1]{forniulcigrai:timechanges}.
\end{remark}


\section{Computation of the push-forwards}\label{s:ODE}

In this section, we compute the push-forward $(\widetilde{h}_t)_{\ast}(W)$ of a left-invariant vector field $W \in \lietre$ via $\widetilde{h}_t$. We recall that the Lie derivative of the vector field $W$ with respect to the vector field $V$ is defined by
\begin{equation}\label{eq:Liebr}
\left( \mathscr{L}_V(W) \right)_p = \frac{\diff}{\diff t} \Big\rvert_{t=0} (\varphi^V_{-t})_{\ast} W_{\varphi^V_t(p)} = \lim_{t\to 0}\frac{(\varphi^V_{-t})_{\ast} W_{\varphi^V_t(p)}  - W_p }{t},
\end{equation}
and coincides with the Lie brackets $[V,W]_p$.

In general, let us write
$$
(\widetilde{h}_t)_{\ast}(W) = \sum_{V \in \mathscr{B}} a_V(t) V
$$
for some functions $a_V(t) \colon \mathcal{M} \to \R$. We remark that 
\begin{equation}\label{eq:1}
\frac{\diff}{\diff t} (a_V(t) \circ \widetilde{h}_t) = \frac{\diff a_V(t)}{\diff t} \circ \widetilde{h}_t + \widetilde{U}a_V(t) \circ \widetilde{h}_t.
\end{equation}
On one hand
\begin{equation}\label{eq:compare1}
\frac{\diff}{\diff t}(\widetilde{h}_t)_{\ast}(W) = \sum_{V \in \mathscr{B}}\frac{\diff}{\diff t} a_V(t) V,
\end{equation}
but also
$$
(\widetilde{h}_{t+s})_{\ast}(W) = \sum_{V \in \mathscr{B}} (a_V(t) \circ \widetilde{h}_{-s} ) (\widetilde{h}_s)_{\ast}(V),
$$
so that, differentiating w.r.t.~$s$ at $s=0$ and by \eqref{eq:Liebr}, we get 
\begin{equation}
\begin{split}
\frac{\diff}{\diff t}(\widetilde{h}_t)_{\ast}(W) &= \sum_{V \in \mathscr{B}} \left( -(\widetilde{U}a_V(t))V + a_V(t) \frac{\diff}{\diff s} \Big\rvert_{s=0} (\widetilde{h}_s)_{\ast}(V) \right) \\
&= \sum_{V \in \mathscr{B}} \left( -(\widetilde{U}a_V(t))V - a_V(t) [ \widetilde{U}, V] \right). \label{eq:compare2}
\end{split}
\end{equation}
Equating the two expressions \eqref{eq:compare1} and \eqref{eq:compare2}, and using \eqref{eq:1}, we obtain
\begin{equation}\label{eq:system}
\sum_{V \in \mathscr{B}} \frac{\diff}{\diff t} (a_V(t) \circ \widetilde{h}_t) V \circ \widetilde{h}_t = \sum_{V \in \mathscr{B}} - (a_V(t) \circ \widetilde{h}_t)  [ \widetilde{U}, V] \circ \widetilde{h}_t,
\end{equation}
which is a system of ODEs. 

\begin{prop}\label{parabolic}
Under the assumption of Theorem \ref{thm:1}, we have that  $\lVert D \widetilde{h}_t \rVert_{\infty}=O(|t|^4)$; hence the flow $\{ \widetilde{h}_t \}_{t \in \R}$ is parabolic (in the sense of Definition \ref{def1}).
\end{prop}
\begin{proof}
By definition, we have that $ [ \widetilde{U}, V] = [U,V] + \beta [Z,V] - (V\beta) Z$ for all $V \in \mathscr{B}$, where $\mathscr{B}$ is the frame chosen in \eqref{eq:thebasis}. Since $U, Z \in \mathfrak{n}$, the operators $\mathfrak{ad}_U = [U, \cdot]$ and $\mathfrak{ad}_Z = [Z, \cdot]$ are nilpotent and in triangular form w.r.t.~the basis~$\mathscr{B}$. The system \eqref{eq:system} is therefore in triangular form and can be solved by substitutions. In particular, for all $V \in \mathscr{B} \setminus \{Z\}$, one can check that the solutions $a_V(t)$ exhibit a polynomial growth in $t$ of order at most $O(|t|^3)$. The only nontrivial equation is in the $Z$-component
$$
\frac{\diff}{\diff t} (a_{Z}(t) \circ \widetilde{h}_t) = (Z \beta \circ \widetilde{h}_t) a_Z(t) \circ \widetilde{h}_t + \alpha(t) \circ \widetilde{h}_t,
$$
for some explicit function $\alpha(t) = O(|t|^3)$. The solution is
$$
a_{Z}(t) \circ \widetilde{h}_t =\exp \Big( \int_0^t Z\beta  \circ \widetilde{h}_{\tau} \diff \tau \Big) \Bigg( \int_0^t (\alpha(\tau)\circ \widetilde{h}_{\tau}) \exp \Big( - \int_0^{\tau} Z\beta  \circ \widetilde{h}_{s} \diff s \Big) \diff \tau + \text{const} \Bigg).
$$
Equation \eqref{eq:vol} can be rewritten as $Z\beta = -\widetilde{U} \log \lambda$; therefore the exponential factor above becomes
$$
\exp \Big( \int_0^t Z\beta  \circ \widetilde{h}_{\tau} \diff \tau \Big) = \exp \Big( \int_0^t \widetilde{U} \log (\lambda^{-1})  \circ \widetilde{h}_{\tau} \diff \tau \Big) = \frac{\lambda}{\lambda \circ \widetilde{h}_t}, 
$$
which implies that $a_Z(t)$ is of order at most $O(|t|^4)$. 
\end{proof}

Recall that there exists $W \in \mathfrak{n} \cap \mathscr{B}$ such that $[U,W]= -c Z$ for some $c \neq 0$.
We are interested in its push-forward.
We have that 
$$
[ \widetilde{U}, W] = [U,W] + \beta[Z,W] - (W\beta) Z = -(c+ W\beta) Z, \text{\ \ \ and\ \ \ } [ \widetilde{U}, Z] = -(Z\beta) Z.
$$
Thus, the system of equations \eqref{eq:system} with the only non zero initial condition $a_W(0) \not\equiv 0$ reduces to a single equation
$$
\frac{\diff}{\diff t} (a_{Z}(t) \circ \widetilde{h}_t) = (Z \beta \circ \widetilde{h}_t) a_Z(t) \circ \widetilde{h}_t + (c + W \beta)\circ \widetilde{h}_t,
$$
whose solution is
$$
a_Z(t) \circ \widetilde{h}_t = \frac{1}{\lambda \circ \widetilde{h}_t} \int_0^t (\lambda \cdot (c+ W\beta))  \circ \widetilde{h}_{\tau} \diff \tau.
$$
Therefore,
\begin{equation}\label{eq:solw}
(\widetilde{h}_t)_{\ast}(W) = W + \Big(\frac{1}{\lambda} \int_{-t}^0 (\lambda \cdot (c + W\beta))  \circ \widetilde{h}_{\tau} \diff \tau\Big) Z.
\end{equation}
Finally, for the push-forward of $Z$, we get
\begin{equation}\label{eq:solz}
(\widetilde{h}_t)_{\ast}(Z) =  \frac{\lambda \circ \widetilde{h}_{-t}}{\lambda} Z.
\end{equation}


\section{Ergodicity and mixing}\label{s:ergodicity}

In this section, under the assumption of Theorem \ref{thm:1}, we prove that the flow  $\{ \widetilde{h}_t \}_{t \in \R}$ is ergodic and, from this, we will deduce it is mixing. Ergodicity is established using a smooth version of Mautner Phenomenon for homogeneous flows. The proof of mixing follows the same ideas as in \cite{forniulcigrai:timechanges} by Forni and Ulcigrai for the case of time-changes; however, their bootstrap argument appears not to be generalizable to our setting, and for this reason the nature of the spectrum of the flow $\{ \widetilde{h}_t \}_{t \in \R}$ remains an open question.

Fix $\sigma > 0$ and, for each $p \in \mathcal{M}$, consider the family of curves
$$
\mathscr{F}_p = \big\{ \{\varphi^{(t)}_s(p)\}_{s \in [0,\sigma]} : t \geq 1\big\}, \text{\ \ \ where\ \ \ } \varphi^{(t)}_s(p) = (\widetilde{h}_{t} \circ \varphi^{\frac{1}{t}W}_s \circ \widetilde{h}_{-t})(p).
$$
The curves $\varphi^{(t)}_s(p)$ for $s \in [0,\sigma]$ start at $p$ and are obtained by pushing segments in direction $W$ of length $\sigma / t$, for $t \geq 1$, via $\widetilde{h}_{t}$.

By the chain rule and equation \eqref{eq:solw}, the vector field inducing $\varphi^{(t)}_s$ is given by
\begin{equation}\label{eq:tgtvect}
\frac{\diff}{\diff s}\Big\rvert_{s=0}(\widetilde{h}_{t} \circ \varphi^{\frac{1}{t}W}_s \circ \widetilde{h}_{-t})(p) = D\widetilde{h}_{t}\Big\rvert_{ \widetilde{h}_{-t}} \Big(\Big( \frac{1}{t} W\Big) \circ \widetilde{h}_{-t}\Big)(p) = (\widetilde{h}_{t})_{\ast}\Big(\frac{1}{t}W\Big) (p) = \frac{1}{t}W + \frac{\ell_t(p)}{\lambda(p)}Z,
\end{equation}
where 
\begin{equation}\label{eq:defvt}
\ell_t(p) = \frac{1}{t}\int_{-t}^0 ( \lambda \cdot (c + W \beta))\circ\widetilde{h}_{\tau}(p) \diff \tau.
\end{equation}
By Birkhoff Theorem, there exists $\ell \in L^1(\mathcal{M})$ such that $\ell_t(p) \to \ell(p)$ for almost every $p \in \mathcal{M}$. 
We remark that, by our assumption \eqref{eq:conditionbeta}, the functions $\ell_t$ are uniformly bounded away from 0, so that $\ell(p) \neq 0$ almost everywhere.

Recall that we defined the time-change $\widetilde{Z} = (1/\lambda) Z$ of $Z$ which commutes with $\widetilde{U}$, see Remark~\ref{rk:comm}.
\begin{prop}\label{thm:ell}
The function $\ell(p)$ is $\varphi^Z_t$-invariant, hence equal to a constant $\ell \neq 0$ almost everywhere. Moreover, for almost every $p \in \mathcal{M}$, we have $\varphi^{(t)}_s(p) \to \varphi^{\ell \widetilde{Z}}_s(p)$ for all $s \in [0,\sigma]$.
\end{prop}
The proof of the Proposition \ref{thm:ell} is postponed to \S\ref{s:lemmas}. 

\begin{prop}\label{ergodic}
The flow $\{ \widetilde{h}_t \}_{t \in \R}$ is ergodic.
\end{prop}
\begin{proof}
Let $s \in \R$.
Since the measures $\widetilde{\omega}=\lambda \omega$ and $\omega$ are equivalent, we have that $f \in  L^2(\mathcal{M},\omega)$ if and only if $f \in  L^2(\mathcal{M},\widetilde{\omega})$, and
$$
(\min \lambda) \norm{f}^2_{ L^2(\mathcal{M},\omega)} \leq \norm{f}^2_{ L^2(\mathcal{M},\widetilde{\omega})} \leq (\max \lambda) \norm{f}^2_{ L^2(\mathcal{M},\omega)}.
$$
In particular, it follows that for any $t \geq 1$, if $f \in  L^2(\mathcal{M},\widetilde{\omega})$, then $f \circ \varphi_s^{(t)} \in  L^2(\mathcal{M},\widetilde{\omega})$, and
$$
 \norm{ f \circ \varphi^{(t)}_s }_{L^2(\mathcal{M},\widetilde{\omega})}^2 = \norm{f\circ \widetilde{h}_t \circ \varphi^{\frac{1}{t}W}_s }^2_{L^2(\mathcal{M},\widetilde{\omega})} \leq (\max \lambda) \norm{f\circ \widetilde{h}_t}^2_{L^2(\mathcal{M},\omega)} \leq \frac{\max \lambda}{\min \lambda} \norm{f}^2_{L^2(\mathcal{M},\widetilde{\omega})}.
$$
Let $f \in  L^2(\mathcal{M},\widetilde{\omega})$ be a real-valued function. Then,
\begin{equation}\label{eq:estimateMP}
\begin{split}
& \modulo{ \int_{\mathcal{M}} f \circ \varphi^{(t)}_s \ \lambda\omega  -\int_{\mathcal{M}} f \ \lambda\omega} = \modulo{ \int_{\mathcal{M}} f \circ \widetilde{h}_t \circ \varphi^{\frac{1}{t}W}_s \circ \widetilde{h}_{-t} \ \lambda\omega -  \int_{\mathcal{M}} f \ \lambda\omega} \\
& \quad  = \modulo{ \int_{\mathcal{M}} f \circ \widetilde{h}_t \circ \varphi^{\frac{1}{t}W}_s\ \lambda\omega - \int_{\mathcal{M}} f\ \lambda\omega } = \modulo{  \int_{\mathcal{M}} (f \circ \widetilde{h}_t) \cdot (\lambda \circ \varphi^{\frac{1}{t}W}_{-s}) \   \omega - \int_{\mathcal{M}} (f \circ \widetilde{h}_t)\ \lambda\omega  } \\
& \quad \leq  \int_{\mathcal{M}} \modulo{f \circ \widetilde{h}_t} \cdot \modulo{\frac{\lambda \circ \varphi^{\frac{1}{t}W}_{-s} - \lambda}{\lambda} } \   \lambda\omega \leq  \norm{f}_1 \cdot \norm{ \frac{\lambda \circ \varphi^{\frac{1}{t}W}_{-s} - \lambda}{\lambda} }_{\infty} \to 0, \text{\ \ \ for } t \to \infty.
\end{split}
\end{equation}

By Proposition \ref{thm:ell}, for almost all $p \in \mathcal{M}$, we have that 
$$
\varphi^{(t)}_s(p) = \widetilde{h}_t \circ \varphi^{\frac{1}{t}W}_s \circ \widetilde{h}_{-t}(p) \to \varphi^{\ell \widetilde{Z}}_s(p).
$$ 
In particular, for any continuous function $f$, by  Lebesgue Theorem and \eqref{eq:estimateMP},
$$
\int_{\mathcal{M}} f \circ  \varphi^{\ell \widetilde{Z}}_s \ \widetilde{\omega} = \int_{\mathcal{M}} \lim_{t \to \infty} f \circ \varphi^{(t)}_s \ \widetilde{\omega} = \lim_{t \to \infty}  \int_{\mathcal{M}} f \circ \varphi^{(t)}_s \ \widetilde{\omega} = \int_{\mathcal{M}} f \ \widetilde{\omega},
$$
which implies that $ \varphi^{\ell \widetilde{Z}}_s$ preserves $\widetilde{\omega}$.
Therefore, by the density of continuous functions in $L^2(\mathcal{M},\widetilde{\omega})$ and the estimate \eqref{eq:estimateMP} above, it follows that $ \norm{ f \circ \varphi^{(t)}_s - f \circ \varphi^{\ell \widetilde{Z}}_s }_2~\to~0$ for all $f \in L^2(\mathcal{M},\widetilde{\omega})$.

The flow $\{ \varphi^{\ell \widetilde{Z}}_s\}_{s \in \R}$ is a time-change of $\{\varphi^Z_s\}_{s \in \R}$, hence is ergodic w.r.t.~$\widetilde{\omega}$. Let us fix $s \in \R$ such that the map $\varphi^{\ell \widetilde{Z}}_s \colon \mathcal{M} \to \mathcal{M}$ is ergodic (indeed, for all but at most countably many $s$, the time-$s$ map $\varphi^{\ell \widetilde{Z}}_s$ is ergodic, see \cite{PughShub}).

Let $g\in L^2(\mathcal{M},\widetilde{\omega})$ be a nonzero, real-valued, $\widetilde{h}_t$-invariant function. 
Cauchy-Schwarz inequality yields
\begin{equation*}
\begin{split}
\norm{g}_2^2 &= \lim_{t \to \infty} \langle g \circ \varphi^{\frac{1}{t}W}_s, g \rangle = \lim_{t \to \infty} \langle g \circ\widetilde{h}_t \circ \varphi^{\frac{1}{t}W}_s, g \circ \widetilde{h}_t \rangle = \lim_{t \to \infty} \langle g \circ \widetilde{h}_t \circ \varphi^{\frac{1}{t}W}_s \circ \widetilde{h}_{-t}, g \rangle  \\
&= \lim_{t \to \infty} \langle g \circ \varphi^{(t)}_s, g \rangle = \langle g \circ \varphi^{\ell \widetilde{Z}}_s, g \rangle \leq \norm{ g \circ \varphi^{\ell \widetilde{Z}}_s}_2 \norm{g}_2 = \norm{g}_2^2.
\end{split}
\end{equation*}
Since the equality holds, $g$ and $g \circ \varphi^{\ell \widetilde{Z}}_s$ are linearly dependent and so we must have $ \xi(s) g = g \circ \varphi^{\ell \widetilde{Z}}_s$ almost everywhere, where $\xi(s) = \pm 1$.
The same argument for $s/2$ gives us $\xi(s/2) g = g \circ \varphi^{\ell \widetilde{Z}}_{s/2}$ almost everywhere; from which we get
$$
\xi(s) g = g \circ \varphi^{\ell \widetilde{Z}}_s = (g \circ \varphi^{\ell \widetilde{Z}}_{s/2} ) \circ \varphi^{\ell \widetilde{Z}}_{s/2} = \xi(s/2) (g \circ \varphi^{\ell \widetilde{Z}}_{s/2}) = \xi(s/2)^2 g.
$$
This yields that $\xi(s) = 1$, i.e., $g$ is invariant under $\varphi^{\ell \widetilde{Z}}_s$. Since $s$ was chosen so that $\varphi^{\ell \widetilde{Z}}_s$ is ergodic, we conclude that $g$ is constant.
\end{proof}

We now show that ergodicity of $\{ \widetilde{h}_t \}_{t \in \R}$ implies it is mixing.
\begin{prop}\label{mixing}
The flow $\{ \widetilde{h}_t \}_{t \in \R}$ is mixing.
\end{prop}
\begin{proof}
By ergodicity, we have that for $\widetilde{\omega}$-a.e.~$p \in \mathcal{M}$,
\begin{equation}\label{eq:2}
v_t(p) := \frac{1}{t}\int_{0}^t (\lambda \cdot (c + W\beta))  \circ \widetilde{h}_{\tau}(p) \diff \tau \to \ell \neq 0.
\end{equation}
Let $f,g \in \mathscr{C}^1(\mathcal{M})$ be smooth functions with $\int_{\mathcal{M}} f \widetilde{\omega} = 0$; we have to show that 
$$
\lim_{t \to \infty} \langle f \circ \widetilde{h}_{t} , g \rangle = \lim_{t \to \infty} \int_{\mathcal{M}} (f \circ \widetilde{h}_{t} ) g\ \lambda\omega =  0.
$$
Fix $\sigma > 0$. We consider again the flow $\{\varphi_s^W\}_{s \in \R}$ generated by $W$. The Haar measure $\omega$ is invariant under $\varphi^W$, hence
$$
 \int_{\mathcal{M}} (f \circ \widetilde{h}_{t} ) g\ \lambda\omega =  \frac{1}{\sigma} \int_0^{\sigma}  \int_{\mathcal{M}} (f \circ \widetilde{h}_{t} \circ \varphi_s^W ) (\lambda g \circ \varphi_s^W) \omega \diff s.  
$$ 
Integration by parts gives
\begin{equation*}
\begin{split}
 & \frac{1}{\sigma} \int_0^{\sigma}  \int_{\mathcal{M}} \Big(f \circ \widetilde{h}_{t} \circ \varphi_s^W \Big) (\lambda g \circ \varphi_s^W) \omega \diff s = \frac{1}{\sigma}   \int_{\mathcal{M}} \Big(\int_0^{\sigma}f \circ \widetilde{h}_{t} \circ \varphi_s^W \diff s \Big) (\lambda g \circ \varphi_{\sigma}^W) \omega \\
 & \qquad -    \frac{1}{\sigma}  \int_0^{\sigma}\int_{\mathcal{M}}  \Big(\int_0^{s}f \circ \widetilde{h}_{t} \circ \varphi_r^W \diff r\Big) (W(\lambda g) \circ \varphi_s^W)\ \omega \diff s.
\end{split}
\end{equation*}
Therefore
\begin{equation*}
\begin{split}
\modulo{ \int_{\mathcal{M}} (f \circ \widetilde{h}_{t} ) g\ \lambda\omega } \leq 
\left(  \frac{1}{\sigma} \norm{\lambda g}_{\infty} + \norm{W(\lambda g) }_{\infty}\right)   \int_{\mathcal{M}} \sup_{s \in [0,\sigma]} \modulo{\int_0^{s}f \circ \widetilde{h}_{t} \circ \varphi_r^W \diff r} \omega.
\end{split}
\end{equation*}
By Lebesgue Theorem, it is enough to show that the last term goes to zero pointwise almost everywhere for $t \to \infty$.

Fix $0 \leq s \leq \sigma$. For any point $p$ and for all $t \geq 1$, let
$$
\gamma(r) = \gamma_{t,p}^{s}(r) := \widetilde{h}_{t} \circ \varphi_r^W(p), \text{\ \ \ for }r \in [0,s];
$$
by \eqref{eq:solw}, the tangent vectors at this curve are
\begin{equation}\label{eq:tggamma}
\frac{\diff}{\diff r} \gamma(r) =( (\widetilde{h}_{t})_{\ast}(W))({\gamma(r)}) = W + \Bigg(\frac{1}{\lambda(\gamma(r))} \int_{0}^t (\lambda \cdot (c + W\beta))  \circ \widetilde{h}_{\tau}( \varphi_r^W(p)) \diff \tau\Bigg) Z. 
\end{equation}

Let $\lambda \widehat{Z}$ be the smooth 1-form dual to the vector field $\widetilde{Z}= \lambda^{-1}Z$. Since
\begin{equation*}
\begin{split}
\frac{1}{t} \int_{\gamma} f\ \lambda \widehat{Z} &= \frac{1}{t} \int_0^{s}(f \circ \widetilde{h}_{t} \circ \varphi^W_r) \Big(\int_{0}^t (\lambda \cdot (c + W\beta))  \circ \widetilde{h}_{\tau}(\varphi_r^W(p)) \diff \tau\Big)\diff r \\
& = \int_0^{s}(f \circ \widetilde{h}_{t} \circ \varphi^W_r) v_t(\varphi_r^W(p)) \diff r,
\end{split}
\end{equation*}
we have
\begin{equation}\label{eq:add}
 \int_0^{s} f \circ \widetilde{h}_{t} \circ  \varphi^W_r \diff r = \frac{1}{\ell \cdot t} \int_{\gamma} f\ \lambda \widehat{Z} + \int_0^{s}(f \circ \widetilde{h}_{t} \circ  \varphi^W_r) \left(1- \frac{v_t(\varphi_r^W(p))}{ \ell } \right) \diff r.
\end{equation}
By ergodicity of $\varphi^Z$, and hence of $\varphi^{\widetilde{Z}}$, we can assume that $f$ is a smooth coboundary for $\varphi^{\widetilde{Z}}$, namely $f = \widetilde{Z}u$ for some $u \in \mathscr{C}^1(\mathcal{M})$. 
For all $V \in \mathscr{B}$, denote by $\widehat{V}$ the smooth 1-form dual to $V$.
Notice that, when integrating $\diff u = \sum_{V \in \mathscr{B}} Vu\ \widehat{V}$ along $\gamma$, the only non zero terms are those corresponding to the components along $W$ and $Z$. 
Thus, by \eqref{eq:tggamma}, we have
$$
\int_{\gamma} \diff u = \int_{\gamma} Zu\ \widehat{Z} + \int_{\gamma} Wu\ \widehat{W} = \int_{\gamma} f\ \lambda\widehat{Z} + \int_{\gamma} Wu\ \widehat{W},
$$
which yields the estimate
$$
\modulo{\int_{\gamma} f\ \lambda \widehat{Z} } \leq \modulo{ \int_{\gamma} \diff u } + \modulo{ \int_{\gamma} Wu\ \widehat{W} } \leq 2 \norm{u}_{\infty} + \norm{Wu}_{\infty} \sigma.
$$ 

Thus, the first integral in the right-hand side of \eqref{eq:add} is uniformly bounded.
Moreover, as we saw in \eqref{eq:2}, for almost every $p \in \mathcal{M}$ for almost every $r \in [0,s]$ we have $v_t(\varphi_r^W(p)) \to \ell$. Therefore
$$
\modulo{\int_0^{s}f \circ \widetilde{h}_{t} \circ \varphi^W_r \diff r} \leq \frac{2 \norm{u}_{\infty} + \norm{Wu}_{\infty} \sigma}{\ell \cdot t} + \norm{f}_{\infty}\int_0^{s} \modulo{1- \frac{v_t(\varphi_r^W(p))}{\ell} } \diff r \to 0 \text{ a.e.,}
$$
again by Lebesgue theorem.
\end{proof}

Theorem \ref{thm:1} follows from Propositions \ref{parabolic}, \ref{ergodic} and \ref{mixing}.


\section{Proof of Proposition \ref{thm:ell}}\label{s:lemmas}

In this section, we prove Proposition \ref{thm:ell} by showing that $\ell$ is constant almost everywhere and, for almost every $p \in \mathcal{M}$, we have $\varphi^{(t)}_s \to \varphi^{\ell \widetilde{Z}}_s$ for all $s \in [0,\sigma]$.

Let us start by some preliminary lemmas.

\begin{lemma}\label{thm:extend}
If, for a given point $p \in \mathcal{M}$, a sequence of curves $\{\varphi^{(n_k)}_s\}_{k \in \N} \subset \mathscr{F}_p$ converges uniformly to a curve $\psi_s(p)$, then for all $r \in \R$ we have $\varphi^{(n_k)}_s \circ \varphi^{\widetilde{Z}}_r(p) \to \varphi^{\widetilde{Z}}_r \circ \psi_s(p)$.
\end{lemma}
Thus, if $\varphi^{(n_k)}_s(p) \to \psi_s(p)$, then for all $q = \varphi^{\widetilde{Z}}_r (p)$ we have that $\varphi^{(n_k)}_s (q) \to \psi_s(q)$, where $\psi_s(q) = \varphi^{\widetilde{Z}}_r \circ \psi_s(p)$. In particular, $\psi_s$ and $\varphi^{\widetilde{Z}}_r$ commute.
\begin{proof}[Proof of Lemma \ref{thm:extend}]
Fix any $R>0$. We show that the tangent vectors of  $\varphi^{(t)}_{s}\circ \varphi^{\widetilde{Z} }_r(p)$ converge uniformly in $r \in [-R,R]$ to $1/\big(\lambda \circ \varphi^{\widetilde{Z} }_r(p) \big) Z$ for $t \to \infty$. Since, by hypothesis, for $r=0$ we have $\varphi^{(n_k)}_s (p) \to \psi_s(p)$, we can conclude that the limit of $\varphi^{(n_k)}_{s}\circ \varphi^{\widetilde{Z} }_r(p)$ exists and is the curve starting at $\psi_s(p)$ with tangent vector $1/\big(\lambda \circ \varphi^{\widetilde{Z} }_r(p) \big) Z$, namely the curve $\varphi^{\widetilde{Z} }_r \circ \psi_s(p)$. The situation is represented in Figure \ref{fig:2}.

\begin{figure}[h!]
\centering
\begin{tikzpicture}[scale=3.5]
\foreach \Point in {(0.5,0.195), (1.1,0.195), (0.5,0.595), (1.1,0.595), (2.2,1.045), (2.8,1.045), (2.95,1.295), (3.55,1.295)}{ 
    \node at \Point {\textbullet};
}
\clip (0,0) rectangle (4,1.5);
\draw[->] (0.5,0.2) -- (0.8,0.2);
\draw (0.8,0.2) -- (1.1,0.2);
\draw[->] (0.5,0.6) -- (0.78,0.6);
\draw[->] (0.78,0.6) -- (0.82,0.6);
\draw (0.82,0.6) -- (1.1,0.6);
\draw (0.5,0.2) -- (0.5,0.6);
\draw (1.1,0.2) -- (1.1,0.6);
\draw[->] (1.25,0.75) sin (1.65,0.95);
\draw[<-] (1.35,0.65) sin (1.75,0.85);
\draw (2,1.05) -- (2.2,1.05);
\draw[->] (2.2,1.05) -- (2.5,1.05);
\draw (2.5,1.05) -- (3.6,1.05);
\draw[->] (2.95,1.3) -- (3.23,1.3);
\draw[->] (3.23,1.3) -- (3.27,1.3);
\draw (3.27,1.3) -- (3.55,1.3);
\draw (2.2,1.05) sin (2.46,1.12) cos (2.77,1.24) sin (2.95,1.3);
\draw (2.8,1.05) sin (3.06,1.12) cos (3.37,1.24) sin (3.55,1.3);
\draw[white] (1.4,0.95)  node[anchor=east] {\textcolor{black}{$\widetilde{h}_{t}$}};
\draw[white] (1.55,0.65)  node[anchor=west] {\textcolor{black}{$\widetilde{h}_{-t}$}};
\draw[black] (0.5,0.2)  node[anchor=north] {$\widetilde{h}_{-t}(p)$};
\draw[white] (0.5,0.6)  node[anchor=south] {\textcolor{black}{$\varphi_s^{\frac{1}{t}W} \circ \widetilde{h}_{-t}(p)$}};
\draw[white] (1.1,0.2)  node[anchor=north west] {\textcolor{black}{$\varphi_s^{\widetilde{Z}} \circ \widetilde{h}_{-t}(p) = \widetilde{h}_{-t} \circ \varphi_s^{\widetilde{Z}}(p) $}};
\draw[black] (2.2,1.05)  node[anchor=north] {$p$};
\draw[black] (2.8,1.05)  node[anchor=north] {$\varphi_s^{\widetilde{Z}}(p)$};
\draw[black] (2.95,1.3)  node[anchor=south] {$ \varphi^{(t)}_s(p)$};
\end{tikzpicture}
\caption{The flows $\{ \varphi^{(t)}_s \}_{s \in [0, \sigma]}$ and $\{ \varphi^{\widetilde{Z}}_r\}_{r \in \R}$.}
\label{fig:2}
\end{figure}

We first compute the push-forward $\big(\varphi^{(t)}_s \big)_{\ast} (\widetilde{Z} ) $ for $t \geq 1$. By Remark \ref{rk:comm}, $\big( \widetilde{h}_{t} \big)_{\ast} (\widetilde{Z} ) = (\widetilde{Z} ) $.
In order to compute the push-forward  $\big(\varphi^{\frac{1}{t}W}_s \big)_{\ast} (\widetilde{Z} )$, we have to solve a system analogous to \eqref{eq:system}. Also in this case, the system is in triangular form, hence the only nontrivial equation is 
\begin{equation*}
\begin{split}
\frac{\diff}{\diff s} \big(a_{\widetilde{Z} }(s) \circ \varphi^{\frac{1}{t}W}_s\big)\widetilde{Z} \circ \varphi^{\frac{1}{t}W}_s  &= -  \big(a_{\widetilde{Z} }(s) \circ \varphi^{\frac{1}{t}W}_s\big)\left[ \frac{1}{t}W, \frac{1}{\lambda} Z \right] \circ \varphi^{\frac{1}{t}W}_s \\
&= \big(a_{\widetilde{Z} }(s) \circ \varphi^{\frac{1}{t}W}_s\big) \frac{1}{t} \frac{W\lambda}{\lambda} \widetilde{Z} \circ \varphi^{\frac{1}{t}W}_s.
\end{split}
\end{equation*}
We get
$$
\big( \varphi^{\frac{1}{t}W}_s \big)_{\ast} (\widetilde{Z} ) =  \exp\Big( \frac{1}{t} \int_{-s}^0 \frac{W\lambda}{\lambda} \circ  \varphi^{\frac{1}{t}W}_{\tau} \diff \tau  \Big)\widetilde{Z}. 
$$
From this, we deduce
\begin{equation*}
\begin{split}
\big(\varphi^{(t)}_s \big)_{\ast} (\widetilde{Z} ) &= \big(\widetilde{h}_{t} \big)_{\ast}\big( \varphi^{\frac{1}{t}W}_s \big)_{\ast}\big( \widetilde{h}_{-t} \big)_{\ast}(\widetilde{Z} )  = \big(\widetilde{h}_{t} \big)_{\ast}\big( \varphi^{\frac{1}{t}W}_s \big)_{\ast} (\widetilde{Z} ) \\
& = \big(\widetilde{h}_{t} \big)_{\ast}\Big( \exp\Big( \frac{1}{t} \int_{-s}^0 \frac{W\lambda}{\lambda} \circ  \varphi^{\frac{1}{t}W}_{\tau} \diff \tau  \Big) \widetilde{Z}  \Big) \\
&= \exp\Big( \frac{1}{t} \int_{-s}^0 \frac{W\lambda}{\lambda} \circ  \varphi^{\frac{1}{t}W}_{\tau} \circ \widetilde{h}_{-t} \diff \tau  \Big) \widetilde{Z}. 
\end{split}
\end{equation*}
For any $s \in [0,\sigma]$ and any initial point $q \in \mathcal{M}$,
$$
\modulo{ \frac{1}{t} \int_{-s}^0 \frac{W\lambda}{\lambda}\circ  \varphi^{\frac{1}{t}W}_{\tau} \circ \widetilde{h}_{-t} (q) \diff \tau } \leq \frac{\sigma}{t} \norm{ \frac{W\lambda}{\lambda}}_{\infty} \to 0, \text{\ \ \ for\ }t \to \infty.
$$
Therefore, 
for any fixed $s \in [0,\sigma]$, the tangent vectors of the curves $\varphi^{(t)}_{s}\circ \varphi^{\widetilde{Z} }_r(p)$ converge uniformly in $r$, that is
$$
\frac{\diff}{\diff r} \big(\varphi^{(t)}_{s}\circ \varphi^{\widetilde{Z}}_r \big) (p) = D\varphi^{(t)}_{s}\Big\rvert_{\varphi^{\widetilde{Z}}_r(p)} \Big( \frac{1}{\lambda} Z\Big)(p) \to \frac{1}{\lambda \circ \varphi^{\widetilde{Z} }_r(p)}Z.
$$ 
Since at the initial point $p$, i.e.~for $r=0$, by hypothesis we have $\varphi^{(n_k)}_s (p) \to \psi_s(p)$, the sequence $\varphi^{(n_k)}_s \circ \varphi^{\widetilde{Z} }_r(p)$ converges to $\varphi^{\widetilde{Z} }_r \circ \psi_s(p)$ uniformly in $r\in [-R,R]$.
\end{proof}

Consider a typical point $p \in \mathcal{M}$. The family 
$$
\mathscr{F}_p = \{ \varphi^{(t)}_s(p) : s \in [0,\sigma] \} \subset \mathscr{C}([0,\sigma], \mathcal{M})
$$
is clearly pointwise relatively bounded. 
For $t \geq 1$ sufficiently large, we have 
\begin{equation}\label{eq:ztbounded}
\modulo{\frac{\diff}{\diff s} \varphi^{(t)}_s(p) } \leq \norm{\frac{1}{t}W + \frac{\ell_t(p)}{\lambda(p)}Z }_{\infty}\leq 1 + \frac{\max \lambda}{\min \lambda} (c+ \norm{W \beta}_{\infty}),
\end{equation}
therefore, $\mathscr{F}_p$ is also equi-Lipschitz. Hence, by Ascoli-Arzel\`{a} Theorem, it is relatively compact in $\mathscr{C}([0,\sigma], \mathcal{M})$. Consider a converging subsequence $\varphi_s^{(n_k)}(p) \to \psi_s(p)$. The limit $\psi_s(p)$ is Lipschitz and, in particular, it is differentiable for almost every $s \in [0,\sigma]$. Since the $W$-component of the tangent vectors of $\varphi^{(n_k)}_s(p)$ converges uniformly to zero by \eqref{eq:tgtvect}, the limit curve $\psi_s(p)$ is parallel to $Z$.
Moreover, by Lemma \ref{thm:extend}, $\psi_s$ is defined for all points in the $Z$-orbit of $p$.
\begin{lemma}\label{thm:llambda}
Let $q \in \mathcal{M}$ be such that $\varphi^{(n_k)}_s(q) \to \psi_s(q)$ for all $s \in [0,\sigma]$. Then, if the tangent vector of $\psi_s$ at $q$ exists, it equals $({\ell}/{\lambda})(q) Z$.
\end{lemma}
In order to prove Lemma \ref{thm:llambda}, we need the following estimates.

\begin{lemma}\label{thm:Zft}
There exist constants $C_Z >0$ and $C_W >0$ such that for all $t \geq 1$ we have $|Z\ell_{t}| \leq C_Z$ and $|W\ell_{t}| \leq C_W t$.
\end{lemma}
\begin{proof}
Define $C_1 = \norm{\lambda \cdot (c+ W \beta)}_{\infty}$, so that for all $t \geq 1$ and for all $p \in \mathcal{M}$ we have $|\ell_t(p)| \leq C_1$, and $C_2= \norm{Z(\lambda \cdot (c+ W \beta))}_{\infty}$.
A direct computation using \eqref{eq:solz} yields
\begin{equation*}
\begin{split}
\modulo{Z\ell_{t}} &= \modulo{\frac{1}{t}\int_{-t}^0 Z \big( ( \lambda \cdot (c+ W \beta))\circ\widetilde{h}_{\tau}\big) \diff \tau }= \modulo{ \frac{1}{t}\int_{-t}^0 (\widetilde{h}_{\tau})_{\ast}(Z)( \lambda \cdot (c + W \beta))\circ\widetilde{h}_{\tau} \diff \tau}\\
&= \modulo{ \frac{1}{t}\int_{-t}^0 \frac{\lambda}{\lambda \circ \widetilde{h}_{\tau} } Z( \lambda \cdot (c + W \beta))\circ\widetilde{h}_{\tau} \diff \tau} \leq \frac{\max \lambda}{\min \lambda} C_2.
\end{split}
\end{equation*}
Similarly, by \eqref{eq:solw},
\begin{equation*}
\begin{split}
\modulo{W\ell_{t}} &= \modulo{ \frac{1}{t}\int_{-t}^0 (\widetilde{h}_{\tau})_{\ast}(W)( \lambda \cdot (c + W \beta))\circ\widetilde{h}_{\tau} \diff \tau}\\
&= \modulo{ \frac{1}{t}\int_{-t}^0 \left(W + \frac{\tau \ell_{\tau}}{\lambda}Z\right)( \lambda \cdot (c + W \beta))\circ\widetilde{h}_{\tau} \diff \tau} \leq \norm{W(\lambda \cdot (c+ W \beta))}_{\infty} + \frac{C_1}{\min \lambda} C_2 \frac{t}{2},
\end{split}
\end{equation*}
which concludes the proof.
\end{proof}

\begin{proof}[Proof of Lemma \ref{thm:llambda}]
Let $f \in \mathscr{C}^{\infty}(\mathcal{M})$; we need to show that 
$$
\frac{\diff}{\diff s}\Big\rvert_{s=0} f \circ \psi_s(q) = \frac{\ell}{\lambda}(q) Zf(q).
$$
whenever the limit exists.
We denote by $\varphi^{(n_k)}(q)$ and $\psi(q)$ the curves $s \mapsto \varphi^{(n_k)}_s(q)$ and $s \mapsto \psi_s(q)$ for $s \in [0,\sigma]$ respectively. 

By Stokes Theorem, since $\varphi^{(n_k)}_0(q) \to \psi_0(q)$ and $\varphi^{(n_k)}_{\sigma}(q) \to \psi_{\sigma}(q)$, we have
\begin{equation}\label{eq:stokes}
\int_{\varphi^{(n_k)}(q)} \diff f \to \int_{\psi(q)} \diff f.
\end{equation}
On the other hand, by \eqref{eq:tgtvect},
$$
\int_{\varphi^{(n_k)}(q)} \diff f = \frac{1}{n_k} \int_0^{\sigma} Wf \circ \varphi_s^{(n_k)}(q) \diff s + \int_0^{\sigma} \left(Zf \cdot \frac{\ell_{n_k}}{\lambda} \right) \circ \varphi^{(n_k)}_s(q) \diff s. 
$$
We rewrite the second term in the right-hand side as
\begin{equation*}
\begin{split}
\int_0^{\sigma} \left(Zf \cdot \frac{\ell_{n_k}}{\lambda} \right) \circ \varphi^{(n_k)}_s(q) \diff s =  & \int_0^{\sigma} \big( Zf \circ \varphi^{(n_k)}_s(q) \big) \cdot \Big(\frac{\ell_{n_k}}{\lambda} \circ \psi_s(q)\Big) \diff s \\
& + \int_0^{\sigma} \big( Zf \circ \varphi^{(n_k)}_s(q) \big) \cdot \Big( \frac{\ell_{n_k}}{\lambda} \circ \varphi^{(n_k)}_s(q) - \frac{\ell_{n_k}}{\lambda} \circ \psi_s(q) \Big) \diff s.
\end{split}
\end{equation*}
By the Mean-Value Theorem, see Figure \ref{fig:3},
$$
\modulo{\frac{\ell_{n_k}}{\lambda} \circ \varphi^{(n_k)}_s(q) - \frac{\ell_{n_k}}{\lambda} \circ \psi_s(q)} \leq \modulo{Z \Big(\frac{\ell_{n_k}}{\lambda}\Big)} \cdot \text{dist}(\varphi^{(n_k)}_s(q), \psi_s(q)) + \modulo{W \Big(\frac{\ell_{n_k}}{\lambda}\Big)} \frac{s}{n_k}. 
$$
\begin{figure}[h!]
\centering
\begin{tikzpicture}[scale=3.5]
\foreach \Point in {(0.2,0.3), (1.8,0.3), (2.3,1.2)}{
    \node at \Point {\textbullet};
}
\clip (-0.5,-0.1) rectangle (3.5,1.5);
\draw (0,0.3) -- (2.5,0.3);
\draw[->] (0,0.7) -- (0,1.1);
\draw[black] (0,0.9)  node[anchor=east] {$W$};
\draw[->] (0.8,0.1) -- (1.2,0.1);
\draw[white] (1,0.1)  node[anchor=north] {\textcolor{black}{$Z$}};
\draw (0.2,0.3) sin (0.6,0.5) cos (1,0.75) sin (1.3,0.9) cos (1.9,1.05) sin (2.3,1.2);
\draw[thick, red, densely dotted] (2.3,1.2) -- (2.3,0.3); 
\draw[thick, red, densely dotted] (1.8,0.27) -- (2.3,0.27); 
\draw[black] (0.2,0.3)  node[anchor=north] {$q$};
\draw[black] (1.8,0.3)  node[anchor=north east] {$\psi_s(q)$};
\draw[black] (2.3,1.2)  node[anchor=south] {$ \varphi^{(n_k)}_s(q)$};
\draw[red] (2.3,0.8)  node[anchor=west] {$ = \frac{s}{n_k}$};
\draw[red] (2,0.25)  node[anchor=north west] {$\leq \text{dist}\big( \varphi^{(n_k)}_s(q), \psi_s(q) \big)$};
\end{tikzpicture}
\caption{Application of the Mean-Value Theorem.}
 \label{fig:3}
\end{figure}
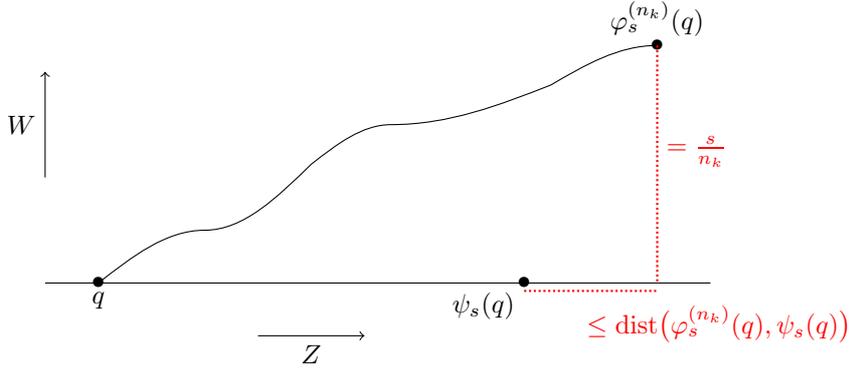

By Lemma \ref{thm:Zft}, there exists a constant $C$ such that
$$
\modulo{\frac{\ell_{n_k}}{\lambda} \circ \varphi^{(n_k)}_s(q) - \frac{\ell_{n_k}}{\lambda} \circ \psi_s(q)} \leq C \big( \text{dist}(\varphi^{(n_k)}_s(q), \psi_s(q)) + s \big),
$$
therefore
\begin{equation*}
\begin{split}
\modulo{\int_{\varphi^{(n_k)}(q)} \diff f - \int_0^{\sigma} \big( Zf \circ \varphi^{(n_k)}_s(q) \big) \cdot \frac{\ell_{n_k}}{\lambda} \circ \psi_s(q) \diff s} \leq & \frac{\norm{Wf}_{\infty} \sigma}{n_k} \\
+  \norm{Zf}_{\infty} C  \int_0^{\sigma} \big( \text{dist}(\varphi^{(n_k)}_s(q), \psi_s(q)) + s \big) \diff s.
\end{split}
\end{equation*}
We remark that $(\ell_t / \lambda)(p)$ is uniformly bounded in $t$ and $p$ as shown in \eqref{eq:ztbounded}. Hence, taking the limit for $k \to \infty$, using \eqref{eq:stokes} and Lebesgue Theorem,
$$
\modulo{\int_{\psi(q)} \diff f - \int_0^{\sigma} \left(Zf \cdot \frac{\ell}{\lambda} \right) \circ \psi_s(q) \diff s} \leq  \norm{Zf}_{\infty} C \frac{\sigma^2}{2}.
$$
Finally, dividing by $\sigma$ and taking the limit $\sigma \to 0$, 
$$
\frac{\diff}{\diff s}\Big\rvert_{s=0} f \circ \psi_s(q) = \lim_{\sigma \to 0} \frac{1}{\sigma} \int_{\psi(q)} \diff f = \lim_{\sigma \to 0} \frac{1}{\sigma}  \int_0^{\sigma} \left(Zf \cdot \frac{\ell}{\lambda} \right) \circ \psi_s(q) \diff s = \frac{\ell}{\lambda}(q)Zf(q).
$$
\end{proof}

We are now in the position to conclude the proof of Proposition \ref{thm:ell}.

\begin{proof}[Proof of Proposition \ref{thm:ell}]
Consider $p \in \mathcal{M}$ and let $\psi_s(p)$ be a limit point of $\mathscr{F}_p$ in $\mathscr{C}([0,\sigma], \mathcal{M})$ as above. 
The family $\{\ell_t\circ \varphi^Z_s(p): t \geq 1\}$ is uniformly bounded and, by Lemma \ref{thm:Zft}, it is equi-Lipschitz in $s$. By Ascoli-Arzel\`{a} Theorem, it is relatively compact and every limit point is a Lipschitz function. Therefore, since $\ell_t \to \ell$ almost everywhere, the function $\ell \circ \varphi^Z_s(p)$ is Lipschitz for almost every $p$; in particular, $Z \ell$ exists almost everywhere. 

We now show that $Z\ell =0$ almost everywhere; since $\ell \circ \varphi^Z_s(p)$ is Lipschitz for almost every $p$, this implies that $\ell$ is constant along almost every $\varphi^Z$-orbit. From the ergodicity of $\{\varphi^Z_t\}_{t \in \R}$, we deduce that $\ell$ is constant almost everywhere. 

By Lemma \ref{thm:extend}, we have $\psi_s (p) =  \varphi^{\widetilde{Z} }_r \circ \psi_s \circ  \varphi^{\widetilde{Z} }_{-r}(p)$.
For almost every $p \in \mathcal{M}$, we can differentiate with respect to $s$, and, by Lemma \ref{thm:llambda}, we obtain
$$
\ell(p) \widetilde{Z} = \left( D \varphi_r^{\widetilde{Z}}\big(\varphi^{\widetilde{Z} }_{-r}(p)\big)\right) \left( \ell \circ \varphi^{\widetilde{Z} }_{-r}(p) \widetilde{Z} \right).
$$
Differentiating with respect to $r$, since $Z\ell$ exists almost everywhere, we deduce
$$
0 = \frac{\diff}{\diff r}\Big\rvert_{r=0} \left( D \varphi_r^{\widetilde{Z}}\big(\varphi^{\widetilde{Z} }_{-r}(p)\big)\right) \left( \ell \circ \varphi^{\widetilde{Z} }_{-r}(p) \widetilde{Z} \right) = - \big( \widetilde{Z} \ell (p)\big) \widetilde{Z} - \ell(p) \mathscr{L}_{\widetilde{Z}}(\widetilde{Z}) = \left( - \frac{1}{\lambda(p)}Z\ell (p) \right) \widetilde{Z}.
$$
This implies that $Z \ell = 0$ almost everywhere, which was our claim.

We obtained that, for almost every $p \in \mathcal{M}$, the tangent vector of $\psi_s$ at $p$ is $\ell \widetilde{Z} $ so that $\psi_s(p) =  \varphi^{\ell \widetilde{Z} }_s(p)$. Since this holds for every limit point $\psi_s(p)$, $\varphi^{\ell \widetilde{Z} }_s(p)$ is the only limit point for $\mathscr{F}_p$, where $p$ was arbitrarily chosen in a full-measure set. 
\end{proof}


\section{Proof of Theorem \ref{thm:conjug}}\label{s:pp7}

We now prove Theorem \ref{thm:conjug}: we show that $\{\widetilde{h}_t\}_{t \in \R}$ is smoothly conjugated to the unperturbed homogeneous flow $\{\varphi_t^U\}_{t \in \R}$ if and only if $\beta$ is a smooth $\widetilde{U}$-coboundary.

Let us assume that there exists a smooth function $w\in \mathscr{C}^{\infty}(\mathcal{M})$ such that $\beta = \widetilde{U}(-w)$. 
We claim that $F(p) = \varphi^Z_{w( p)}(p)$ realizes the conjugacy $F \circ \widetilde{h}_t = \varphi^U_t \circ F$.
In order to prove this, we compute the push-forward of $\widetilde{U}$ by $F$ and we show it equals $U$, namely $DF(\widetilde{U}) = U \circ F$.
For any smooth function $f$ and any point $p \in \mathcal{M}$, by the chain rule, we have
\begin{equation*}
\begin{split}
[DF(\widetilde{U})](f)(p) &= \widetilde{U}(f \circ F)(p) = \frac{\diff}{\diff t} \Big\rvert_{t=0} f \circ F \circ \widetilde{h}_t (p) = \frac{\diff}{\diff t} \Big\rvert_{t=0} f \circ \varphi^Z_{w( \widetilde{h}_t (p) )} \circ \widetilde{h}_t (p) \\
& = ((Zf \circ F) \widetilde{U}w)(p) + [D\varphi^Z_{w(p)}(\widetilde{U})](f)(p).
\end{split}
\end{equation*}
Since $[U,Z] = [Z,Z] = 0$, we deduce that 
$$
D\varphi^Z_{w(p)}(\widetilde{U}) = D\varphi^Z_{w(p)}(U + \beta Z) = U \circ \varphi^Z_{w(p)} + \beta \cdot (Z \circ \varphi^Z_{w(p)}).
$$
Therefore, since $\beta = \widetilde{U}(-w)$, we conclude
$$
[DF(\widetilde{U})](f) =  (Zf \circ F) \widetilde{U}w + Uf \circ F + \beta \cdot (Zf \circ F) = Uf \circ F,
$$
which proves our claim.


Conversely, let us assume that there exists a $\mathscr{C}^1$-diffeomorphism $F$ such that $F \circ \widetilde{h}_t = \varphi^U_t \circ F$. 
Since $[\widetilde{U}, \widetilde{Z}]=0$, the push-forward $V:=F_{\ast}(\widetilde{Z})$ of $\widetilde{Z}$ commutes with $U =  F_{\ast}(\widetilde{U})$. 
The proof follows three steps: first, we show in Lemma \ref{th:vising} that $V$ is a left-invariant vector field; from this we deduce that $V$ is a constant multiple of $Z$, see Lemma \ref{th:visz}. Finally, we prove in Lemma \ref{th:isom_implies_cob} that this implies that $\lambda$ is cohomologous to a constant in $L^2$, and, exploiting a result by Wang \cite[Theorem 2.1]{wang}, we deduce that the transfer function is smooth. 

\begin{lemma}\label{th:vising}
The vector field $V$ is left-invariant, that is $V \in \lietre$.
\end{lemma}
\begin{proof}
Let us write $V = \sum_{E \in \mathscr{B}} a_E \cdot E$, where $\mathscr{B}$ is the frame in \eqref{eq:thebasis} and $a_E \colon \mathcal{M} \to \R$ are smooth functions. We will prove that they are constant. Indeed, since $V$ commutes with $U$, we have
\begin{equation*}
\begin{split}
0 = & [V, U] = \left[\sum_{E \in \mathscr{B} } a_E \cdot E,\ c_{1,2}E_{1,2} + c_{2,3} E_{2,3} + c_{1,3} E_{1,3} \right] = (-U a_{E_{3,1}} ) E_{3,1} \\
& + (- Ua_{E_{2,1}} - c_{2,3} \cdot a_{E_{3,1}}) E_{2,1} + (- Ua_{E_{3,2}} + c_{1,2} \cdot a_{E_{3,1}} ) E_{3,2} \\
& + (-Ua_{\frac{1}{2}(E_{1,1}- E_{2,2})} - c_{1,2} \cdot a_{E_{2,1}} - c_{1,3} \cdot a_{E_{3,1}}) \frac{1}{2}(E_{1,1}-E_{2,2}) \\
& + (-Ua_{\frac{1}{2}(E_{2,2}- E_{3,3})} - c_{2,3} \cdot a_{E_{3,2}} - c_{1,3} \cdot a_{E_{3,1}}) \frac{1}{2}(E_{2,2}-E_{3,3}) \\
& + \left(- Ua_{E_{1,2}} - c_{1,3} \cdot a_{E_{3,2}} + c_{1,2} \left( 2 a_{\frac{1}{2}(E_{1,1}- E_{2,2})} - a_{\frac{1}{2}(E_{2,2}- E_{3,3})} \right) \right) E_{1,2} \\
& + \left(- Ua_{E_{2,3}} + c_{1,3} \cdot a_{E_{2,1}} + c_{2,3} \left( 2 a_{\frac{1}{2}(E_{2,2}- E_{3,3})} - a_{\frac{1}{2}(E_{1,1}- E_{2,2})}\right) \right) E_{2,3} \\
& + \left(- Ua_{E_{1,3}} - c_{1,2} \cdot a_{E_{2,3}} + c_{2,3} \cdot a_{E_{1,2}} + c_{1,3} \left(a_{\frac{1}{2}(E_{1,1}- E_{2,2})} + a_{\frac{1}{2}(E_{2,2}- E_{3,3})} \right) \right) E_{1,3}.
\end{split}
\end{equation*}
All the coefficients in brackets are equal to zero, in particular, from the first one, we get $U a_{E_{3,1}} = 0$. By ergodicity of $U$, we deduce that $a_{E_{3,1}}$ is constant. Considering the second term in brackets, we obtain 
$$
Ua_{E_{2,1}} = - c_{2,3} \cdot a_{E_{3,1}} = \text{const},
$$
from which we deduce $a_{E_{3,1}} = 0$ and $a_{E_{2,1}}$ is constant. Proceeding in this way for all the remaining terms, we conclude that $a_E$ is constant for all $E \in \mathscr{B}$; that is, $V \in \lietre$.
\end{proof}

The next step, Lemma \ref{th:visz} below, exploits the notion of Kakutani equivalence. We recall that two measurable flows are \newword{Kakutani equivalent} if there exists a time-change of one which is isomorphic to the other. 

\begin{lemma}\label{th:visz}
We have that $V= F_{\ast}(\widetilde{Z}) = aZ$ for some constant $a \neq 0$.
\end{lemma}
\begin{proof}
We remark that the time-change $\{ \varphi^{\widetilde{Z}}_t \}_{t \in \R}$ is parabolic in the sense of Definition \ref{def1}; therefore $\{ \varphi^V_t \}_{t \in \R}$ must be parabolic as well. Since, by Lemma \ref{th:vising}, $V \in \lietre$ is a homogeneous vector field, it follows that $V$ is $\mathfrak{ad}$-nilpotent. 
It is possible to check explicitly that, if $c_{1,2} \cdot c_{2,3} \neq 0$, then $V \in \langle U, Z \rangle$; namely, the only $\mathfrak{ad}$-nilpotent $V$\rq{}s which commute with $U$ are linear combinations of $U$ and $Z$. The only other possibilities are when $U = c_{1,2} E_{1,2}$, in which case $V$ could be a multiple of $E_{3,2}$, or when $U = c_{2,3} E_{2,3}$, in which case $V$ could be a multiple of $E_{2,1}$.

Since the time-change $\widetilde{Z}$ is smoothly conjugated to $V$, the homogeneous flow induced by $Z$ is Kakutani equivalent to the one induced by $V$. In \cite{kanigowskietal}, the authors introduce an invariant $e( \cdot, \log)$ for Kakutani equivalence for unipotent flows and they provide an explicit formula to compute it in terms of the Jordan block structure of the adjoint matrix of their generating vector fields. It is an easy computation in our case to determine the Kakutani invariant $e( \{\varphi^V_t\}_{t \in \R}, \log)$ for any $V \in \lietre$ of the possibilities listed above, using \cite[Theorem 1.1]{kanigowskietal}. The Kakutani invariant for $Z$ is
$$
e( \{\varphi^Z_t\}_{t \in \R}, \log) = GR(Z) - 3 = 5-3 = 2;
$$
while, for all other cases above with $V \notin \langle Z \rangle$, we have $e( \{\varphi^V_t\}_{t \in \R}, \log)  > 2$. Therefore, we conclude that $V$ is a multiple of $Z$. 
\end{proof}

\begin{lemma}\label{th:isom_implies_cob}
If $\widetilde{Z} = \frac{1}{\lambda} Z$ is smoothly conjugate to $aZ$, with $a \neq 0$, then $\lambda - 1$ is a $L^2$-coboundary, namely there exists $w \in L^2(\mathcal{M})$ such that $\lambda -1 = Zw$.
\end{lemma}
\begin{proof}
Let us denote by $\widetilde{\varphi}_t$ the flow induced by $\widetilde{Z} = \frac{1}{\lambda} Z$. 
By assumption there exists a diffeomorphism $F \colon \mathcal{M} \to \mathcal{M}$ such that for every $p \in \mathcal{M}$ and for every $t \in \R$ 
$$
 \widetilde{\varphi}_t \circ F (p) = F \circ \varphi^{aZ}_t(p).
$$
Considering the differentials, and using $\varphi^{aZ}_t = \varphi^{Z}_{at}$, we have
\begin{equation}\label{eq:1}
D \widetilde{\varphi}_t (F(p)) \cdot DF(p) \cdot \left( D\varphi^Z_{at}(p) \right)^{-1} = DF(\varphi^Z_{at}(p)).
\end{equation}
Notice that $\left( D\varphi^Z_{at}(p) \right)^{-1}  = D\varphi^Z_{-at}(\varphi^Z_{at}(p))$.

We write all the matrices with respect to the frame $\mathscr{B}$ introduced in \eqref{eq:thebasis}.
Denote by $f_{i,j}(p)$ the entries of the $8 \times 8$ matrix $DF(p)$.
Since $F$ is a $\mathscr{C}^1$-diffeomorphism and $\mathcal{M}$ is compact, each entry $f_{i,j}(\varphi^Z_{at}(p))$ of $DF(\varphi^Z_{at}(p))$ is uniformly bounded in $t$ and~$p$.

Standard computations give us
$$
D\varphi^Z_{-at}(\varphi^Z_{at}(p)) = 
\begin{pmatrix}
1 & \ & \ & \ & \ & \ & \ & \ \\
0 & 1 & \ & \ & \ & \ & \ & \ \\
0 & 0 & 1 & \ & \ & \ & \ & \ \\
2at & 0 & 0 & 1 & \ & \ & \ & \ \\
2at & 0 & 0 & 0 & 1 & \ & \ & \ \\
0 & 0 & at & 0 & 0 & 1 & \ & \ \\
0 & -at & 0 & 0 & 0 & 0 & 1 & \ \\
-(at)^2 & 0 & 0 & -\frac{at}{2} & -\frac{at}{2} & 0 & 0 & 1 
\end{pmatrix},
$$
so that the $j$-th row $R_j(p)$ of the product $DF(p) \cdot D\varphi^Z_{-at}(\varphi^Z_{at}(p))$ equals
$$
R_j(p) = \mathbf{e}_j^T \cdot DF(p) \cdot D\varphi^Z_{-at}(\varphi^Z_{at}(p)) = 
\begin{pmatrix}
f_{j,1}(p) + 2at(f_{j,4}(p) +f_{j,5}(p)) - (at)^2 f_{j,8}(p) \\
f_{j,2}(p) - at f_{j,7}(p) \\
f_{j,3}(p) + atf_{j,6}(p) \\
f_{j,4}(p) - \frac{at}{2} f_{j,8}(p) \\
f_{j,5}(p) - \frac{at}{2} f_{j,8}(p) \\
f_{j,6}(p) \\
f_{j,7}(p) \\
f_{j,8}(p)
\end{pmatrix}^T,
$$
where $(\cdot)^T$ denotes the transpose.

We now compute the matrix $D \widetilde{\varphi}_t$. As we did in \S\ref{s:ODE}, let us write 
$$
\left( \widetilde{\varphi}_t\right)_{\ast} (A) = \sum_{E \in \mathscr{B}} a_E(t)  E.
$$ 
In order to determine the functions $a_E(t)$, we have to solve the system of ODEs
$$
\sum_{E \in \mathscr{B}} \frac{\diff}{\diff t} (a_E(t) \circ  \widetilde{\varphi}_t) E \circ  \widetilde{\varphi}_t = \sum_{E \in \mathscr{B}} - (a_E(t) \circ  \widetilde{\varphi}_t)  \left[ \frac{1}{\lambda}Z, E \right] \circ  \widetilde{\varphi}_t,
$$
namely
\begin{equation*}
\begin{split}
& \frac{\diff}{\diff t} (a_{ E_{3,1}}(t) \circ  \widetilde{\varphi}_t) = 0 \\
& \frac{\diff}{\diff t} (a_{E_{2,1}}(t) \circ  \widetilde{\varphi}_t) = 0 \\
& \frac{\diff}{\diff t} (a_{E_{3,2}}(t) \circ  \widetilde{\varphi}_t) = 0 \\
& \frac{\diff}{\diff t} (a_{E_{1,1} - E_{2,2}}(t) \circ  \widetilde{\varphi}_t) = - \frac{a_{ E_{3,1}}(t)}{ \lambda } \circ  \widetilde{\varphi}_t\\
& \frac{\diff}{\diff t} (a_{E_{2,2} - E_{3,3}}(t) \circ  \widetilde{\varphi}_t) = - \frac{a_{ E_{3,1}}(t) }{ \lambda  } \circ  \widetilde{\varphi}_t \\
& \frac{\diff}{\diff t} (a_{E_{1,2}}(t) \circ  \widetilde{\varphi}_t) = - \frac{a_{ E_{3,2}}(t) }{ \lambda } \circ  \widetilde{\varphi}_t \\
& \frac{\diff}{\diff t} (a_{E_{2,3}}(t) \circ  \widetilde{\varphi}_t) = \frac{a_{ E_{2,1}}(t) }{ \lambda } \circ  \widetilde{\varphi}_t \\
& \frac{\diff}{\diff t} (a_Z(t) \circ  \widetilde{\varphi}_t) = \frac{a_{E_{1,1} - E_{2,2}} + a_{E_{2,2} - E_{3,3}} }{ \lambda } \circ  \widetilde{\varphi}_t+ \sum_{V \in \mathscr{B}}  a_V(t) \circ  \widetilde{\varphi}_t V\left( \frac{1}{\lambda \circ  \widetilde{\varphi}_t}\right).\\
\end{split}
\end{equation*}
If we denote by 
$$
\Lambda_t(p)= \int_0^t \frac{1}{\lambda} \circ \widetilde{\varphi}_{\tau}(p) \diff \tau,
$$
we can write the matrix $D \widetilde{\varphi}_t$ as
$$
D \widetilde{\varphi}_t(q) =
\begin{pmatrix}
1 & \ & \ & \ & \ & \ & \ & \ \\
0 & 1 & \ & \ & \ & \ & \ & \ \\
0 & 0 & 1 & \ & \ & \ & \ & \ \\
-\Lambda_t(q) & 0 & 0 & 1 & \ & \ & \ & \ \\
-\Lambda_t(q) & 0 & 0 & 0 & 1 & \ & \ & \ \\
0 & 0 & -\Lambda_t(q)& 0 & 0 & 1 & \ & \ \\
0 & \Lambda_t(q) & 0 & 0 & 0 & 0 & 1 & \ \\
\ast & \ast & \ast & \ast & \ast & \ast & \ast & \ast 
\end{pmatrix}.
$$

We compute the first row of the matrix product on the left hand-side of \eqref{eq:1}: since the first row of  $D \widetilde{\varphi}_t$ is $\mathbf{e}_1^T$, it equals 
\begin{equation*}
R_1(p) = 
\begin{pmatrix}
f_{1,1}(p) + 2at(f_{1,4}(p) +f_{1,5}(p)) - (at)^2 f_{1,8}(p) \\
f_{1,2}(p) - at f_{1,7}(p) \\
f_{1,3}(p) + atf_{1,6}(p) \\
f_{1,4}(p) - \frac{at}{2} f_{1,8}(p) \\
f_{1,5}(p) - \frac{at}{2} f_{1,8}(p) \\
f_{1,6}(p) \\
f_{1,7}(p) \\
f_{1,8}(p)
\end{pmatrix}^T
\end{equation*}
As we remarked, each entry has to be bounded uniformly in $t$ for all $p$. Therefore, we deduce that $f_{1,6} \equiv f_{1,7} \equiv f_{1,8} \equiv 0$ and $f_{1,4} \equiv - f_{1,5}$.
Since $DF(p)$ is invertible, for each $p \in \mathcal{M}$, at least one between $f_{1,1}(p), \dots, f_{1,5}(p)$ is not zero.

We now compute the 4th row of the matrix product on the left hand-side of \eqref{eq:1}: it equals 
\begin{equation*}
R_4(p) - \Lambda_t(q) R_1(p) = 
\begin{pmatrix}
f_{4,1}(p) + 2at(f_{4,4}(p) +f_{4,5}(p)) - (at)^2 f_{4,8}(p)  - f_{1,1}(p) \cdot \Lambda_t(q)\\
f_{4,2}(p) - at f_{4,7}(p) - f_{1,2}(p) \cdot \Lambda_t(q) \\
f_{4,3}(p) + at f_{4,6}(p) - f_{1,3}(p) \cdot \Lambda_t(q) \\
f_{4,4}(p) - \frac{at}{2} f_{4,8}(p) - f_{1,4}(p) \cdot \Lambda_t(q) \\
f_{4,5}(p) - \frac{at}{2} f_{4,8}(p) - f_{1,5}(p) \cdot \Lambda_t(q) \\
f_{4,6}(p) \\
f_{4,7}(p) \\
f_{4,8}(p)
\end{pmatrix}^T
\end{equation*}
Since $| \Lambda_t(q) | \leq \text{const} \cdot t$, we must have $f_{4,8} \equiv 0$.
Moreover, since for each $p \in \mathcal{M}$ at least one between $f_{1,1}(p), \dots, f_{1,5}(p)$ is not zero, by looking at the corresponding entry above, we deduce that there exists a constant $K >0$ such that, for each $p \in \mathcal{M}$, we have
$$
\left\lvert \Lambda_t(q) - c(p) t \right\rvert \leq K.
$$
By ergodicity of $\widetilde{\varphi}_t$, for almost all $q=F(p) \in \mathcal{M}$ we have 
$$
c(p) = \int_{\mathcal{M}} \frac{1}{\lambda}\ \lambda \omega = 1,
$$
and the integral 
$$
\Lambda_t(q) - t = \int_0^t \left(\frac{1}{\lambda}-1\right) \circ \widetilde{\varphi}_{\tau}(q) \diff \tau
$$
is uniformly bounded by $K$. By the $L^2$-version of Gottschalk-Hedlund (see, e.g., \cite[Lemma 5.7]{flaminioforni}), it follows that $\frac{1}{\lambda} -1 $ is a $L^2$-coboundary for $\frac{1}{\lambda} Z$, or, equivalently, there exists $w \in L^2(\mathcal{M})$ such that $\lambda - 1 = Zw$.
\end{proof}

We can now conclude the proof of Theorem \ref{thm:conjug}. By Lemma \ref{th:isom_implies_cob}, there exists $w \in L^2(\mathcal{M})$ such that $\lambda -1 = Zw >-1$. By \cite[Theorem 2.1]{wang} and by Sobolev Embedding Theorem, we deduce that $w$ is a smooth function on $\mathcal{M}$. By \eqref{eq:vol},
$$
0 = U\lambda + Z(\beta \lambda) = UZw + Z(\beta \lambda) = Z( Uw + \beta \lambda)= Z ( Uw + \beta + \beta Zw);
$$
therefore, by ergodicity of $\{\varphi^Z_t\}_{t \in \R}$, we conclude that 
$$
\beta = -Uw - \beta Zw = \widetilde{U}(-w).
$$


\subsection*{Acknowledgments} I am grateful to Giovanni Forni for many enlightening suggestions. I would also like to thank Danijela Damjanovic, Livio Flaminio, Vinay Kumaraswamy, and my supervisor Corinna Ulcigrai for several helpful conversations. 
I thank the referees for their valuable comments and suggestions on a previous version of the paper.
The research leading to these results has received funding from the European Research Council under the European Union Seventh Framework Programme (FP/2007-2013)/ERC Grant Agreement n.~335989.

\end{document}